\documentclass[a4paper,12pt]{amsart}
\usepackage{amsmath,amsfonts,amssymb,amsthm,latexsym,amscd}
\usepackage[dvips]{graphicx}
\usepackage[T1]{fontenc}

\usepackage{parskip}

\newcommand{\B}[1]{{\mathbf #1}}
\newcommand{\C}[1]{{\mathcal #1}}

\newtheorem{thm}{Theorem}[section]
\newtheorem{thm*}{Theorem}
\newtheorem{lem}[thm]{Lemma}
\newtheorem{prop}[thm]{Proposition}
\newtheorem{cor}[thm]{Corollary}
\newtheorem{defn}[thm]{Definition}

\newtheorem*{q*}{Question}

\theoremstyle{definition}
\newtheorem{ex}[thm]{Example}

\newtheorem{rem}[thm]{Remark}
\newtheorem*{rem*}{Remark}
\newtheorem*{rems*}{Remarks}
\newtheorem*{cor*}{Corollary}

\def\B{\mathbf}

\newcommand{\Mo}{(M,\omega )}

\newcommand\Diff{\operatorname{Diff}}

\newcommand\Ham{\operatorname{Ham}}

\newcommand\Aut{\operatorname{Aut}}

\newcommand\OP{\operatorname}

\def\C{\mathcal{C}}
\def\D{\mathbf{D}}
\def\Diff{\operatorname{Diff}}
\def\G{\mathcal{G}}
\def\fG{\mathfrak{G}}
\def\fL{\mathfrak{L}}
\def\cH{\mathcal{H}}
\def\Id{\operatorname{Id}}
\def\Im{\operatorname{Im}}

\def\X{\operatorname{X}}

\def\a{\alpha}
\def\b{\beta}
\def\area{\operatorname{area}}

\def\inc{\operatorname{inc}}

\def\g{\gamma}

\def\s{\sigma}

\begin{document}

\title[The autonomous metric]{On the autonomous metric on the group of
area-preserving diffeomorphisms of the 2-disc}
\author{Michael Brandenbursky}
\author{Jarek K\k{e}dra}


\begin{abstract}
Let $\D^2$ be the open unit disc in the Euclidean plane and let
$\B G:=\Diff(\D^2,\area)$ be the group of smooth compactly supported
area-preserving diffeomorphisms of $\D^2$.
For every  natural number $k$ we construct an injective homomorphism
$\B Z^k\to\B G$, which is bi-Lipschitz with respect to the word metric on
$\B Z^k$ and the autonomous metric on $\B G$.
We also show that the space of homogeneous quasi-morphisms vanishing on all
autonomous diffeomorphisms in the above group is infinite dimensional.
\end{abstract}

\maketitle

\section{Introduction}

\label{S:intro}

\subsection{The main result}

Let $\B D^2\subset \B R^2$ be the open unit disc and let
$H\colon \B D^2\to \B R$ be a smooth compactly supported function.
It defines a vector field
$$
X_H(x,y)=
-\frac{\partial H}{\partial y}\partial_x
+\frac{\partial H}{\partial x}\partial_y
$$
which is tangent to the level sets of $H$. Let $h$ be the
time-one map of the flow $h_t$ generated by $X_H$. The diffeomorphism $h$ is
area-preserving and every diffeomorphism arising in this way is called
{\em autonomous}. Such a diffeomorphism is relatively easy to understand
in terms of its generating function.

It is a well known fact that every smooth compactly supported and area-preserving
diffeomorphism  of the disc $\B D^2$ is a composition of finitely
many autonomous diffeomorphisms \cite{Ba}.
How many? In the present paper we are interested in the geometry of this question.
More precisely, we define the {\em autonomous norm}
on the group $\B G:=\Diff(\B D^2,\area)$ of smooth compactly supported
area-preserving diffeomorphisms of the disc by
$$
\|f\|_{\OP{Aut}}:=
\min\left\{m\in \B N\,|\,f=h_1\cdots h_m
\text{ where each } h_i \text{ is autonomous}\right\}.
$$
The associated metric is defined by ${\bf
d}_{\OP{Aut}}(f,g):=\|fg^{-1}\|_{\OP{Aut}}$.  Since the set of autonomous
diffeomorphisms is invariant under conjugation the autonomous metric is
bi-invariant.

\begin{thm*}\label{T:main}
For every natural number $k\in \B N$ there exists an
injective homomorphism $\B Z^k\to \Diff(\B D^2,\area)$ which is bi-Lipschitz
with respect to the word metric on $\B Z^k$ and the autonomous metric on
$\Diff(\B D^2,\area)$.
\end{thm*}

\subsection{Remarks}
\hfill
\begin{enumerate}
\item
We show in the proof of Theorem \ref{T:main} that the embedding of
$\B Z^k$ is constructed to be in the kernel of the Calabi
homomorphism $\C \colon \B G\to \B R$ (see Remark \ref{rem:Calabi} for
a definition and \cite{Ba,C,MS} for more information).

\item
Gambaudo and Ghys defined in \cite[Section 6.3]{GG} the autonomous
metric on the group of area-preserving diffeomorphisms of the 2-sphere and
showed that its diameter is infinite.

\item
For each $p\in[1,\infty)$ the group $\B G$ may be equipped with the right-invariant $L^p$-metric, see Arnol'd-Khesin for a detailed exposition \cite{AK}. Results similar to ours with respect to the $L^p$-metric were obtained by Benaim-Gambaudo in \cite{BG} and by the first author in \cite{B-infinity}.

\item
In a greater generality, the autonomous metric is defined on the group
$\Ham(M,\omega)$ of compactly supported Hamiltonian diffeomorphisms of a
symplectic manifold. It is interesting to know if such a metric is always
unbounded.

\item
Even more generally the autonomous metric can be defined as follows. A
compactly supported diffeomorphism $h$ of a manifold $M$ is called autonomous
if it is the time-one map of a time-independent compactly supported flow $h_t$.
The group $\Diff(\B D^2)$ of all smooth compactly
supported diffeomorphisms of the disc is generated by autonomous diffeomorphisms
and hence the autonomous metric can be defined as above.
However, it is known due to Burago, Ivanov and Polterovich
\cite{BIP}, that the autonomous metric is bounded in this case.

\item
Since the autonomous metric is bi-invariant, investigating
geometric properties of embeddings of non-abelian groups
has to be done with respect to some bi-invariant metrics.
At the time of writing this paper such metrics are not
well understood. In Section \ref{sec-relation} we prove an algebraic
result which may indicate some good geometric properties.
More precisely, for a non-abelian free group of rank two
we construct an injective homomorphism
$\B F_2\to \B G$ and prove that the image of the
induced homomorphism $Q(\B G) \to Q(\B F_2)$ on the
spaces of homogeneous quasi-morphisms is
infinite dimensional.

\end{enumerate}

\subsection{Comments on the proof of Theorem \ref{T:main}}
\label{subsec-comments}
Let us start with a definition.
A function $\psi\colon \Gamma \to~\B R$ from a group $\Gamma$ to the reals is called a {\em quasi-morphism}
if there exists a real number $A\geq 0$ such that
$$
|\psi(gh) - \psi(g) - \psi(h)|\leq A
$$
for all $g,h\in \Gamma$. The infimum of such numbers $A$ is called the
\emph{defect} of $\psi$ and is denoted by $D_\psi$.
If $\psi(g^n)=n\psi(g)$ for all $n\in \B Z$
and all $g\in \Gamma$ then $\psi$ is called \emph{homogeneous}. Any
quasi-morphism $\psi$ can be homogenized by setting

\begin{equation*}\label{eq:definition-homogenization}
\overline{\psi} (g) := \lim_{p\to +\infty} \frac{\psi (g^p)}{p}.
\end{equation*}
The vector space of homogeneous quasi-morphisms on $\Gamma$
is denoted by $Q(\Gamma)$. For more details about quasi-morphisms see e.g. \cite{Cal}.

The first part of the proof is to show that the space $Q(\B G,\OP{Aut})$ of
homogeneous quasi-morphisms on $\Diff(\B D^2,\area)$ which are trivial on the
set of autonomous diffeomorphisms is infinite dimensional.  This is done by
constructing (for $n\geq 3$) an injective linear map
$$\G_n \colon Q(\B B_n,\B A_n)\to Q(\B G,\OP{Aut}),$$
where $\B B_n$ denotes the braid group on
$n$-strings and $\B A_n\subset~\B B_n$ is a certain abelian subgroup defined in
Section~\ref{sec-GG-construction}.

\begin{rem}\label{R:unbounded} Notice that the existence of a nontrivial
homogeneous quasi-morphism $\psi\colon \B G\to \B R$ that is trivial on
$\OP{Aut}\subset \B G$ implies that the autonomous norm is unbounded. Indeed,
for every $f\in G$
we have that $|\psi(f)|=|\psi(h_1\ldots h_m)|\leq mD_\psi$ and hence for every
natural number $n$ we get $\|f^n\|_{\OP{Aut}}\geq \frac{|\psi(f)|}{D_\psi}n>0$,
provided $\psi(f)\neq 0$.  \end{rem}

The map $\G_n$ is defined in Section \ref{sec-GG-construction} and is induced
from the construction due to Gambaudo and Ghys \cite{GG}. The fact that the
quasi-morphism $\G_n(\varphi)$ is trivial on the set of autonomous
diffeomorphisms provided $\varphi$ is trivial on $\B A_n$ is proved in Section
\ref{sec-proofs}. The latter proof consists of two steps. First, we show that
$\G_n(\varphi)$ is trivial on autonomous diffeomorphisms generated by certain
Morse-type functions (Theorem \ref{thm:Morse-type}). Secondly, the set of Morse-type functions is dense in the set of all functions with respect to the $C^1$-topology, and $$\G_n(\varphi)\colon \B G\to \B R$$
is a continuous function, see Theorem \ref{T:Morse-Ham}.

It is known that the space $Q(\B B_n,\B A_n)$ is infinite dimensional (see
Section \ref{subsec-proof-T-Q-Aut}) and hence we obtain that
$Q(\B G,\OP{Aut})$ is infinite dimensional.
Let $\varphi_i\colon \B B_3\to \B R$ be
homogeneous quasi-morphisms comprising a set of $k$ linearly independent
elements of $Q(\B B_3)$. It follows that the map
$\Phi\colon\B G \to \B R^k$ defined by
$$
\Phi(f) = (\G_3(\varphi_1)(f),\ldots,\G_3(\varphi_k)(f)),
$$
is Lipschitz and its image is quasi-isometric to the whole $\B R^k$.

The second part of the proof is a construction of a homomorphism
$\B Z^k\to \B G$ with the required properties.
It is in fact a section of the map
$\Phi\colon\B G\to \B R^k$ mentioned above.  It is defined by constructing $k$
diffeomorphisms $f_j\in \B G$ with disjoint supports (hence commuting)
such that $\G_3(\varphi_i)(f_j)=\delta_{ij}$, where $\delta_{ij}$ is the Kronecker delta.

\begin{rem}
The above map $\Phi$ is a quasi-morphism, if one defines an $\B R^k$ valued quasi-morphism
analogously to the real valued one using  an $L^p$-norm.
Observe that there exists a nontrivial homogeneous
quasi-morphism $\B F_2\to \B R^k$ on the free group of rank two with the image
quasi-isometric to $\B R^k$ for every $k\in \B N$. This follows from the fact that $Q(\B F_2)$ is
infinite dimensional. However, none of such quasi-morphisms admits
a homomorphic section over $\B Z^k$ for $k\geq 2$.
\end{rem}

\section{The Gambaudo-Ghys construction}
\label{sec-GG-construction}
Let us recall a construction, due to Gambaudo and Ghys \cite[Section 5.2]{GG},
which produces a quasi-morphism on $\B G$ from a quasi-morphism on the
pure braid group~$\B P_n$.

Let $g_t\in \B G$ be an isotopy from the identity to $g\in \B G$
and let $z\in \B D^2$ be a basepoint. For $y\in \B D^2$ we define a loop
$\gamma_{y,z}\colon [0,1]\to \B D^2$ by
$$
\gamma_{y,z}(t):=
\begin{cases}
(1-3t)z+3ty &\text{ for } t\in \left [0,\frac13\right ]\\
g_{3t-1}(y) &\text{ for } t\in \left [\frac13,\frac23\right ]\\
(3-3t)g(y)+(3t-2)z & \text{ for } t\in \left [\frac23,1\right ].
\end{cases}
$$

Let $\X_n(\B D^2)$ be the configuration space of all ordered $n$-tuples
of pairwise distinct points in the disc $\B D^2$. It's fundamental group
$\pi_1(\X_n(\B D^2))$ is identified with the pure braid group $\B P_n$.
Let $z=(z_1,\ldots,z_n)$ in $\X_n(\B D)$ be a base point.
For almost every $x=(x_1,\ldots,x_n)\in \X_n(\B D^2)$ the
$n$-tuple of loops $(\gamma_{x_1,z_1},\ldots,\gamma_{x_n,z_n})$ is
a based loop in the configuration space $\X_n(\B D^2)$.
Since the group $\B G$ is contractible (see e.g. \cite[Corollary 2.6]{T})
the based homotopy
class of this loop does not depend on the choice of the
isotopy $g_t$. Let $\gamma(g,x)\in \B P_n=\pi_1(\X_n(\B D^2),z)$
be an element represented by this loop.

Let $\varphi\colon \B P_n\to \B R$ be a homogeneous quasi-morphism.
Define the quasi-morphism
$\Phi_n\colon \B G\to \B R$ and its homogenization
$\overline{\Phi}_n\colon \B G\to \B R$
by
\begin{equation}\label{eq:GG}
\Phi_n(g):=
\int\limits_{\X_n(\D^2)}\varphi(\g(g;{x}))d{x}\qquad\qquad
\overline{\Phi}_n(g):=\lim_{p\to +\infty}\frac{\Phi_n(g^p)}{p}
\thinspace .
\end{equation}

\begin{rem}
The assertion that both the above functions are well defined quasi-morphisms is
proved in \cite[Lemma 4.1]{B}.  Using the family of signature quasi-morphisms
on $\B P_n$ (one for each $n$) Gambaudo-Ghys showed that
$\dim(Q(\B G))=\infty$. This fact was also proved in \cite{BEP}.
\end{rem}

\begin{rem}\label{rem:Calabi}
The Calabi homomorphism $\C\colon\B G\to\B R$ may be defined as follows:
\begin{equation*}\label{eq:Calabi}
\C(g)=\int_0^1\int_{\D^2}H(x,t)dxdt,
\end{equation*}
where $H(x,t)$ defines a flow whose time-one map is $g$, see e.g.
\cite[Lemma 10.27]{MS}.
The group $\B P_2$ is infinite cyclic, hence every homogeneous quasi-morphism
$\varphi_2\colon\B P_2\to\B R$ is a homomorphism. Since the kernel of the
Calabi homomorphism $\C$ is a simple group \cite{Ba}, we have that
$\overline{\Phi}_2(g)=C\cdot\C(g)$ for every $g\in\B G$, where $C$ is a real
constant independent of $g$. The proof of this equality which does not rely on the theorem of Banyaga can be found in \cite{GG1}.
\end{rem}

The above construction defines a linear map $Q(\B P_n)\to Q(\B G)$.  Let
$$
\G_n\colon Q(\B B_n)\to Q(\B G)
$$
be its composition with the homomorphism
$Q(\iota)\colon Q(\B B_n)\to Q(\B P_n)$ induced
by the inclusion $\iota\colon \B P_n\to \B B_n$.
Let $\B A_n\subset \B B_n$ be an abelian group generated by
braids $\eta_{i,n}$ shown in Figure \ref{fig:braids-eta-i-n}.
Recall that
$Q(\B B_n,\B A_n)$ denotes the space of homogeneous quasi-morphisms
on $\B B_n$ that are trivial on $\B A_n$ and that $Q(\B G,\OP{Aut})$
denotes the space of homogeneous quasi-morphisms on $\B G$
that are trivial on autonomous diffeomorphisms.

\begin{thm}\label{T:Q-Aut}
Let $n\geq 3$. The image of the linear map
$$
\G_n\colon Q(\B B_n,\B A_n) \to Q(\B G,\OP{Aut})
$$
is infinite dimensional. In particular, the diameter
of $(\B G,{\bf d}_{\OP{Aut}})$ is infinite.
\end{thm}

\begin{figure}[htb]
\centerline{\includegraphics[height=1.4in]{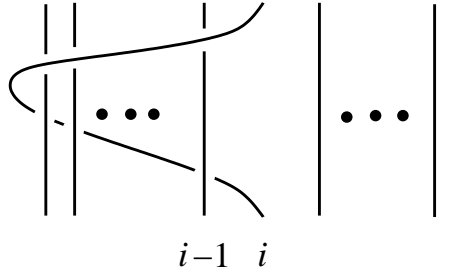}}
\caption{\label{fig:braids-eta-i-n} Braid $\eta_{i,n}$}
\end{figure}

\begin{rem}
Theorem \ref{T:Q-Aut} answers the following question posed to the first
author by L.Polterovich.

{\em
Does there exist a
quasi-morphism on $\B G$, given by the Gambaudo-Ghys construction, which
vanishes on all autonomous diffeomorphisms in $\B G$? In other words, does
there exist a nontrivial element in $\Im(\G_n)$ which vanishes on all
autonomous diffeomorphisms?
}
\end{rem}

\section{Proofs}
\label{sec-proofs}

\subsection{Evaluation of the map $\G_n$ on autonomous diffeomorphisms}
\label{subsec-GG-aut}
Denote the space of autonomous compactly supported Hamiltonians $H\colon\D^2\to\B R$ by $\cH$.

\begin{defn}\label{defn:Morse-type-functions}\rm
We say that a  function $H\in\cH$ is of  \textit{Morse-type}  if:
\begin{enumerate}
\item
The boundary of the support of $H$ is a simple closed curve.
\item
The function $H$ has no degenerate critical points in the interior of its support.
\item
If $x,y$ are two distinct non-degenerate critical points of $H$ then $H(x)\neq H(y)$.
\end{enumerate}
\end{defn}

\begin{thm}\label{thm:Morse-type}
If $\varphi_n\in Q(B_n,A_n)$ then $\overline{\Phi}_n(h)=0$
for every autonomous diffeomorphism $h$ generated by a Morse-type function $H$.
\end{thm}

\begin{proof}
The statement follows from \cite[Theorem 4.5]{B}. More precisely, it is shown there that
\begin{equation*}\label{eq:general-formula}
\overline{\Phi}_n(h):=
\G_n(\varphi_n)(h)=\int\limits_{T}\hbar^\prime d\mu,
\end{equation*}
where $T$ is the Reeb graph of $H$, and $\hbar^\prime\colon T\to \B R$
is a function induced by $H$, and $d\mu$ is a measure on $T$.
All the objects above are defined in \cite[Section 4.2]{B}.
In particular, it follows from the definition of the measure $d\mu$ that it is
trivial if $\varphi_n(\eta_{i,n})=0$ for each $i$. Hence $\overline{\Phi}_n(h)=0$, and the proof follows.
\end{proof}

\begin{rem}
The idea of the proof relies on the fact that $n$ points on different level curves trace a braid conjugate to a braid in $\B A_n$.
\end{rem}

\subsection{The continuity of the Gambaudo-Ghys quasi-morphisms}
\label{SS:continuity}
The aim of this section is to prove the following result which
will be used in the proof of Theorem \ref{T:Q-Aut}.

\begin{thm}\label{T:Morse-Ham}
Let $H\in\cH$ and $\{H_k\}_{k=1}^\infty$ be a sequence of functions such that
each $H_k\in\cH$ and $H_k\xrightarrow[k\rightarrow\infty]{}H$ in
$C^1$-topology. Let $h_1$ and $h_{1,k}$ be the time-one maps of the Hamiltonian flows
generated by $H$ and $H_k$ respectively. Then
$$
\lim\limits_{k\to\infty}\overline{\Phi}_n(h_{1,k})=\overline{\Phi}_n(h_1).
$$
\end{thm}

The proof is presented below as a sequence of assertions and the theorem follows immediately from
Proposition \ref{P:Morse-Ham}.

Let $g\in \B G$ and $\{g_t\}_{t=0}^1\in\B G$ such that $g_0=\Id$ and $g_1=g$.
We denote
$$
\fG(\{g_t\}):=
\int\limits_{\D^2\times\D^2}\frac{1}{2\pi}
\int\limits_0^1 \left\|
\frac{\partial}{\partial t} \left(\frac{g_t(x)-g_t(y)}{\|g_t(x)-g_t(y)\|}
\right)\right\|dt dx dy,
$$
where $\|\cdot\|$ is the Euclidean norm. Lemma 1 in \cite{GL} shows that
$\fG(\{g_t\})$ is well defined. Denote
$$
\fG(g):=\inf_{g_t} \fG(\{g_t\}),
$$
where the infimum is taken over all isotopies $g_t\in \B G$ joining the
identity with $g$. Lemma 2 in \cite{GL} states that for all
$g$ and $h$ in $\B G$
$$
\fG(gh)\leq \fG(g)+\fG(h).
$$
Thus the following limit exists:
$$
\fL(g)=\lim\limits_{n\rightarrow \infty}\frac{\fG(g^n)}{n}.
$$

\begin{prop}\label{thm:crossing-number-inequality}
Let $g\in \B G$, and let $\varphi_n\colon\B B_n\rightarrow\B R$
be a homogeneous quasi-morphism. Then
$$\left|\overline{\Phi}_n(g)\right|\leq C_1 \fL(g),$$
where $C_1>0$ is independent of $g$.
\end{prop}
\begin{proof}
The proof of this proposition is very similar to the proof of Lemma 4 in \cite{BG}. Let $x=(x_1,\ldots,x_n)\in X_n(\D^2)$ and $g_t\in \B G$ joining the
identity with $g$. Following \cite[Section 4.1]{BG}, to every $t\in[0,1]$ and $1\leq i,j\leq n$ we associate the unit vector
$$u(t,x_i,x_j)=\frac{g_t(x_i)-g_t(x_j)}{\|g_t(x_i)-g_t(x_j)\|}\thinspace.$$
Consider a map $u_{x_i,x_j}\colon[0,1]\to\B S^1$ where $t\to u(t,x_i,x_j)$. The change of variables induced by the map $u_{x_i,x_j}$ leads to the equality
\begin{equation}\label{eq:inequality_G}
\fG(\{g_t\})=
\int\limits_{\D^2\times\D^2}\frac{1}{2\pi}
\int\limits_{\B S^1}\#\{u^{-1}_{x_i,x_j}(\omega)\} d\omega dx_i dx_j,
\end{equation}
where $\#$ stands for cardinality. Consider the braid $\gamma(g;x)$
defined by the isotopy $g_t$. Let $S\subset \B S^1$ be the set of all points
$\omega$ in the circle for which the projection onto the line orthogonal to
$\omega$ is injective on the set $\{x_1,...,x_n\}$.  The number of times the
$i$-th strand overcrosses the $j$-th strand is bounded by
$\#\{u^{-1}_{x_i,x_j}(\omega)\}+4$, where $\omega \in \B S^1\setminus S$. The
constant $4$ comes from the fact that we close a path $g_t(x)$. Let
$l(\g(g;x))$ denote the length of the braid $\g(g;x)$ with respect to the Artin
generators $\{\s_i\}_{i=1}^{n-1}$ of $\B B_n$. Note that the measure of the set
$S$ is zero. It follows that for a generic $\omega\in\B S^1$, i.e. for
$\omega\in \B S^1\setminus S$ we have
$$
l(\g(g;x))\leq
\displaystyle\sum_{i\neq j}^{n}(\#\{u^{-1}_{x_i,x_j}(\omega)\}+4),
$$
see also the proof of Lemma 4 in \cite{BG}. Consequently, we have that
\begin{equation}\label{eq:inequality1}
l(\g(g;x))\leq \sum_{i\neq j}^n
\int_{\B S^1}(\#\{u^{-1}_{x_i,x_j}(\omega)\}+4)d\omega.
\end{equation}
Note that since $\varphi_n$ is a homogeneous quasi-morphism there exists a
positive constant $A$ such that $|\varphi_n(\a)|\leq A\cdot l(\a)$ for each
$\a\in\B B_n$.
Let $u_{p,x_i,x_j}$ be the above function corresponding to
the diffeomorphism $g^p$.
It follows from \eqref{eq:inequality_G} and \eqref{eq:inequality1} that
\begin{eqnarray*}
\left|\overline{\Phi}_n(g)\right|&\leq&
A\lim_{p\rightarrow\infty}\int\limits_{\X_n(\D^2)}\frac{l(\g(g^p;{x}))}{p}d{x}\\
&\leq&A\lim_{p\rightarrow\infty}\int\limits_{\X_n(\D^2)}
\frac{\sum_{i\neq j}^n\int_{\B S^1}(\#\{u_{p,x_i,x_j}^{-1}(\omega)\}+4)d\omega}{p}d{x}\\
&\leq& (2\pi)^{n-2}A \lim_{p\to \infty} \frac{1}{p}\sum_{i\neq j}^n
\int\limits_{\D^2\times \D^2}\int\limits_{\B S^1} (\#\{u_{p,x_i,x_j}^{-1}(\omega)\}+4)
d\omega dx_i dx_j\\
&\leq& (2\pi)^{n-2}A \lim_{p\to \infty} \frac{1}{p}\sum_{i\neq j}^n
2\pi \frak{G}(\{g_{t,p}\}),\\
\end{eqnarray*}
where $g_{t,p}$ is any isotopy from the identity to $g^p$. Since the above inequalities hold for any isotopy between the identity and $g^p$ we have
\begin{equation*}
\left|\overline{\Phi}_n(g)\right|\leq(2\pi)^{n-1}n(n-1)A \lim_{p\to \infty} \frac{\frak{G}(g^p)}{p}
\leq (2\pi)^{n-1}n(n-1)\fL(g)
\end{equation*}
and the proof follows.
\end{proof}

Now we will recall a definition of the right-invariant $L^2$-metric on $\B G$.
It is defined as follows. Let $$
L_2\{g_t\}:=\int_0^1 dt \left(\int_{\D^2}\|\dot{g_t}(x)\|^2 dx \right)^{\frac 12}
$$
be the $L^2$-length of a smooth isotopy $\{g_t\}_{t\in [0,1]}\subset\B G$, where
$\|\dot{g_t}(x)\|$ denotes the Euclidean length of the tangent vector
$\dot{g_t}(x)\in T_x\D^2$. Observe that this length is right-invariant, that
is, $L_2\{g_t\circ f\}=L_2\{g_t\}$ for any $f\in \B G$. It defines a
non-degenerate right-invariant metric on $\B G$ by
$$
{\bf d}_2(g_0,g_1):=\inf_{g_t}L_2\{g_t\},
$$
where the infimum is taken over all
paths from $g_0$ to $g_1$. See Arnol'd-Khesin \cite{AK} for a detailed discussion.

\begin{cor}\label{cor:distance-inequality}
Let $g\in\B G$, and let $\varphi\colon\B B_n\to\B R$
be a homogeneous quasi-morphism. Then
$$|\overline{\Phi}_n(g)|\leq C_3 \B d_2(\Id,g),$$
where $C_3>0$ does not depend on $g$.
\end{cor}
\begin{proof}
It follows from \cite[Theorem 1]{GL} that for any  $ g\in\B G$ there exists a
universal constant $C>0$, such that $\fL(g)\leq C \B d_2(\Id,g)$. Now take
$C_3=C\cdot C_1$, and the statement follows from
Proposition~\ref{thm:crossing-number-inequality}.
\end{proof}

\begin{lem}\label{lem:epsilon-close}
Let $F\in\cH$. Then for any $\varepsilon>0$ and $p\in\B N$ there exists $\delta_p>0$, such that if
$H\in\cH$ is $\delta_p$-close to $F$ in $C^1$-topology, then
$$\B d_2(f_1^p,h_1^p)<\varepsilon,$$
where $f_t$ and $h_t$ are the Hamiltonian flows generated by $F$ and $H$.
\end{lem}
\begin{proof}
For the convenience we  normalize the area of $\D^2$  to be 1. It is enough to
show that for all $p\in\B N$ there exists $\delta_p>0$ such that
$$
\max\limits_{x\in\D^2}\|\nabla F_{(x)}-\nabla H_{(x)}\|
<\delta_p\Rightarrow\B d_2(f_1^p,h_1^p)<\frac{\varepsilon}{2}.
$$
Note that
$\B d_2(f_1^p,h_1^p)=\B d_2(\Id,f_1^ph_1^{-p})\leq L_2(\{f_t^ph_t^{-p}\})$.
It follows from \cite[Proposition 1.4.D]{P-book} that
\begin{equation*}
\frac{\partial (f_t^ph_t^{-p})}{\partial t}(x)=p\cdot (X_F-X_{(Hf_t^{-p})})_{(f_t^ph_t^{-p}(x))}\thinspace.
\end{equation*}
Thus
\begin{align*}
\left\|\frac{\partial (f_t^ph_t^{-p})}{\partial t}(x)\right\|
&=p\cdot \left\|(X_F-X_{(Hf_t^{-p})})_{(f_t^ph_t^{-p}(x))}\right\|\\
&=p\cdot \left\|(\nabla F-\nabla(Hf_t^{-p}))_{(f_t^ph_t^{-p}(x))}\right\|,
\end{align*}
Note that $f_t$ is an autonomous Hamiltonian flow. Thus $Ff_t(x)=F(x)$
for all $x\in\D^2$ and $t\in\B R$. It follows that
$\forall x\in\D^2$ and $\forall p\in \B Z$
$$
\nabla F_{(h_t^{-p}(x))}(Df_t^{-p})_{(f_t^ph_t^{-p}(x))}
=\nabla F_{(f_t^ph_t^{-p}(x))}.
$$
We get the following inequality:
$$
\B d_2(f_1^p,h_1^p)
\leq p\int\limits_0^1\left(\int\limits_{\D^2}\|(Df_t^{-p})_{(f_t^ph_t^{-p}(x))} \|_M^2\cdot
\|\nabla F_{(h_t^{-p}(x))}-\nabla H_{(h_t^{-p}(x))}\|^2 dx\right)^{\frac{1}{2}} dt,
$$
where $\|\cdot\|_M$ is a matrix norm. Denote by
$$
\mathfrak{M}_{Df_t}:=
\max\limits_{x\in\D^2,t\in[0,1]}\|(Df_t^{-1})_{(x)}\|_M
\quad\textrm{and}\quad
\delta_p:=\frac{\varepsilon}{2p(\mathfrak{M}_{Df_t})^p}.
$$
We get the following inequality
$$
\B d_2(f_1^p,h_1^p)\leq p(\mathfrak{M}_{Df_t})^p
\max\limits_{x\in\D^2}\|\nabla F(x)-\nabla H(x)\|\leq\frac{\varepsilon}{2}.
$$
It follows that if $H$ is $\delta_p$-close to $F$ in $C^1$-topology,
then $\B d_2(f_1^p,h_1^p)<~\varepsilon$.
\end{proof}

\begin{prop}\label{P:Morse-Ham}
Let $F\in\cH$. Then for any $\varepsilon>0$ there exists $\delta>0$,
such that if $H\in\cH$ is $\delta$-close to $F$ in $C^1$-topology then:
$$
\left|\overline{\Phi}_n(f)-\overline{\Phi}_n(h)\right|\leq\varepsilon,
$$
where $f$ and $h$ are time-one maps of flows generated by $F$ and $H$.
\end{prop}

\begin{proof}
Fix some $\varepsilon>0$. Let $D_{\overline{\Phi}_n}$ be the defect of the
homogeneous quasi-morphism $\overline{\Phi}_n\colon\B G\to\B R$, and let $C_3$
be the constant which was defined in  Corollary \ref{cor:distance-inequality}.
Take $p\in\B N$ such that $\frac{D_{\overline{\Phi}_n}+C_3}{p}<\varepsilon$. It
follows from Lemma \ref{lem:epsilon-close} that there exists $\delta_p>0$, such
that if $H$ is $\delta_p$-close to $F$ in $C^1$-topology, then
$\B d_2(f^p,h^p)<1$. Thus we obtain
$$
\left|\overline{\Phi}_n(f)-\overline{\Phi}_n(h)\right|=
\frac{1}{p}\left|\overline{\Phi}_n(f^p)-\overline{\Phi}_n(h^p)\right|\leq
\frac{D_{\overline{\Phi}_n}+\left|\overline{\Phi}_n(f^ph^{-p})\right|}{p}.
$$
It follows from Corollary \ref{cor:distance-inequality} that
$$
\left|\overline{\Phi}_n(f^ph^{-p})\right|\leq
C_3\B d_2(Id,f^ph^{-p})=
C_3\B d_2(f^p,h^p)<C_3.
$$
Thus
$$
\left|\overline{\Phi}_n(f)-\overline{\Phi}_n(h)\right|<
\frac{D_{\overline{\Phi}_n}+C_3}{p}<\varepsilon.
$$
\end{proof}

\subsection{Proof of Theorem \ref{T:Q-Aut}}
\label{subsec-proof-T-Q-Aut}

Let $n\geq 3$ and denote by
$$
\B A_n:=\langle \eta_{i,n}\,|\,2\leq i\leq n\rangle,
$$
the abelian subgroup of $\B P_n$ generated by braids $\eta_{i,n}$ shown in
Figure~\ref{fig:braids-eta-i-n}. Let $Q(\B B_n, \B A_n)$ be the space of
homogeneous quasi-morphisms on $\B B_n$ which are identically zero on the group
$\B A_n$. It follows from \cite[Theorem 12]{BF} that the space $Q(\B B_n)$ is
infinite dimensional. The restriction of every homogeneous quasi-morphism on an
abelian group is a homomorphism, hence the space $Q(\B B_n, \B A_n)$ is also
infinite dimensional. The following theorem was proved by Ishida, see
\cite[Theorem 1.2]{I}.
\begin{thm}\label{T:GG-injectivity}
The map $\G_n\colon Q(\B B_n)\to Q(\B G)$ is injective.
\end{thm}
In particular the
map $\G_n\colon Q(\B B_n, \B A_n)\to Q(\B G)$ is also injective. It follows
from \cite[Theorem 2.7]{Mil} that Morse-type Hamiltonians form a $C^1$-dense
subset of $\cH$, hence Theorem \ref{thm:Morse-type} and Theorem \ref{T:Morse-Ham}
imply that the image of $Q(\B B_n, \B A_n)$ under the map $\G_n$ lies in
$Q(\B G, \Aut)$ and the proof follows.
\qed

\subsection{Proof of Theorem \ref{T:main}}
\label{subsec-proof-T-main}
Let us start with the following basic result which
is interesting in its own right.

\begin{lem}\label{L:abstract}
Let $\Gamma$ be a group and let $V\subset Q(\Gamma)$ be
a $k$-dimensional subspace, where $k\in \B N$. There exist elements
$g_1,\ldots,g_k\in \Gamma$ and $\psi_1,\ldots,\psi_k\in V$
such that
$$
\psi_i(g_j)=\delta_{ij},
$$
where $\delta_{ij}$ is the Kronecker delta.
\end{lem}

\begin{proof}
We prove the statement by induction on the dimension.
For $k=1$ it is clearly true.
Let $V\subset Q(\Gamma)$ be a $k$-dimensional subspace.
According to the induction hypothesis there exist
$\varphi_i\in V$ and $g_j\in \Gamma$ such that
$\varphi_i(g_j)=\delta_{ij}$, where $1\leq i,j\leq k-1$.

Let $\varphi_k\in V$ be such that $\varphi_1,\ldots,\varphi_k$
are linearly independent.
Let $\psi_k := \varphi_k - \sum_{i=1}^{k-1}\varphi_k(g_i)\varphi_i$.
We get that $\psi_k(g_j) = 0$ for $j=1,\ldots,k-1$.
The above linear independence implies (possibly after
multiplying $\psi_k$ by a nonzero constant) that there
exists $g_k\in \Gamma$ such that $\psi_k(g_k) = 1$.

Let $\psi_i = \varphi_i - \varphi_i(g_k)\psi_k$,
where $1\leq i\leq k-1$.  We obtain that $\psi_i(g_k) = 0$ and
$\psi_i(g_j) = \delta_{ij}$ for each $1\leq i,j\leq k-1$.
Thus the $k$-tuples $\psi_1,\ldots,\psi_k$ and $g_1,\ldots,g_k$
satisfy the statement of the lemma.
\end{proof}

Now we start proving the theorem. Let $r=\frac{1}{k}$. Denote by $\B D_r$ an open disc in the Euclidean
plane of radius $r$ centered at zero. Let
$$\B G_r:=\Diff(\B D_r,\area)$$
be
the group of smooth compactly supported area-preserving diffeomorphisms of
$\B D_r$. Gambaudo-Ghys construction is valid in the case of $G_r$ as well, i.e.
every homogeneous quasi-morphism $\varphi_n\colon\B B_n\to~\B R$ defines a
homogeneous quasi-morphism $\overline{\Phi}_{n,r}\colon\B G_r\to\B R$. This
construction defines a homomorphism $\G_{n,r}\colon Q(\B B_n)\to Q(\B G_r)$.

The vector space
\begin{equation*}\label{eq:im-GG-map}
\Im\left(\G_{n,r}|_{\thinspace Q(\B B_n, \B A_n)}\right)\subset Q(\B G_r, \Aut)
\end{equation*}
is infinite dimensional for $n\geq 3$. The proof of this fact
is identical to the proof of Theorem \ref{T:Q-Aut}. As an immediate consequence of Lemma \ref{L:abstract}
we have the following fact: for each $n\geq 3$ there exist
$\{g_{i,n}\}_{i=1}^k\in \B G_r$ and
$\{\overline{\Phi}_{i,n,r}\}_{i=1}^k\in\Im\left(\G_{n,r}|_{\thinspace Q(\B B_n, \B A_n)}\right)\subset Q(\B G_r, \Aut)$
such that
$$
\overline{\Phi}_{i,n,r}(g_{j,n})=\delta_{ij},
$$
where $\delta_{ij}$ is the Kronecker delta.

We extend every diffeomorphism in $\B G_r$ by identity on the unit disc $\D^2$
and get an injective homomorphism $i_r\colon \B G_r\to \B G$.

\begin{lem}\label{lem:Phi-3-extension}
The following identity holds on the space $Q(\B B_3,\B A_3)$
$$
\G_{3,r} = Q(i_r)\circ \G_3.
$$
Equivalently,  for each
$\overline{\Phi}_{3,r}\in\Im
\left(\G_{3,r}|_{\thinspace Q(\B B_3, \B A_3)}\right)
\subset Q(\B G, \Aut)$ and $g\in~\B G_r$
we have
$$
\overline{\Phi}_{3,r}(g)=\overline{\Phi}_3(i_r(g)),
$$
where
$\overline{\Phi}_3$ is defined using the same quasi-morphism $\varphi_3$ in
$Q(\B B_3, \B A_3)$.  \end{lem}

\begin{proof}
Denote by $\X_3(\B D_r)$ the space of all ordered $3$-tuples of
distinct points in $\B D_r$. It follows that
\begin{eqnarray*}
\overline{\Phi}_3(i_r(g))
&=&\lim_{p\to +\infty}\left(\int\limits_{\X_3(\B D_r)}
\frac{\varphi_3(\g(g^p;x))}{p}\thinspace dx
+\int\limits_{\X_3(\D^2)\setminus\X_3(\B D_r)}
\frac{\varphi_3(\g(g^p;x))}{p}\thinspace dx\right)\\
&=&\overline{\Phi}_{3,r}(g)
+\int\limits_{\X_3(\D^2)\setminus\X_3(\B D_r)}\lim_{p\to +\infty}
\frac{\varphi_3(\g(g^p;x))}{p}
\thinspace dx\hspace{2mm}.
\end{eqnarray*}
Recall that by definition $i_r(g)=\Id$ on $\D^2\setminus\B D_r$. It follows
that for $x\in\X_3(\D^2)\setminus\X_3(\B D_r)$ the braid
$$
\g(g^p;x)=
\a_{1,p,x}\circ\b_{p,x}\circ\eta_{2,3}^{m_{x,p}}
\circ\b_{p,x}^{-1}\circ\a_{2,p,x},
$$
where the length of the braids
$\a_{1,p,x}$ and $\a_{2,p,x}$ is bounded for all $p$ and $x$. It follows that
for all $x\in\X_3(\D^2)\setminus\X_3(\B D_r)$ we have
$$
\lim_{p\to+\infty}\frac{\varphi_3(\g(g^p;x))}{p}=
\lim_{p\to+\infty} \frac{\varphi_3(\b_{p,x}\circ\eta_{2,3}^{m_{x,p}}\circ\b_{p,x}^{-1}))}{p}
=0,
$$
where the last equality follows from the fact that homogeneous quasi-morphisms
are invariant under conjugation and that $\varphi_3(\eta_{2,3})=0$, because
$\varphi_3\in Q(\B B_3, \B A_3)$. Hence
$$
\int\limits_{\X_3(\D^2)\setminus\X_3(\B D_r)}\lim_{p\to
+\infty}\frac{\varphi_3(\g(g^p;x))}{p}\thinspace dx=0,
$$
and the proof of the lemma follows.
\end{proof}

\begin{rem}
Let $\inc\colon\B B_{n-1}\to\B B_n$ be the standard inclusion of braid groups.
A homogeneous quasi-morphism $\varphi_n\in Q(\B B_n)$ is called \emph{kernel
quasi-morphism} if $\varphi_n(\a)=0$ for each $\a\in\Im(\inc)$. In the proof of
Lemma \ref{lem:Phi-3-extension} we used the fact that the space of kernel
quasi-morphisms on $\B B_3$ contains the space $Q(\B B_3, \B A_3)$, i.e. we used
the fact that $\varphi_3(\eta_{2,3})=~0$. If we replace the space
$Q(\B B_n, \B A_n)$ by the space of kernel quasi-mor\-phisms on $\B B_n$, then Lemma
\ref{lem:Phi-3-extension} will hold for $n>3$. In what follows we use the fact
that the space $Q(\B B_3, \B A_3)$ is infinite dimensional, and it is not known
what is the dimension of the space of kernel quasi-morphisms on $\B B_n$ for
$n>3$ (for more information about kernel quasi-morphisms see \cite{Mal}), hence
we restrict ourselves to the case $n=3$.
\end{rem}

Let us proceed with the proof of Theorem \ref{T:main}. For $1\leq j\leq k$
denote by $g_j:=i_r(g_{j,3})\in\B G$. It follows from Lemma
\ref{lem:Phi-3-extension} that
\begin{equation}\label{eq:delta-ij}
\overline{\Phi}_{i,3}(g_j)=\delta_{ij},
\end{equation}
where
$\overline{\Phi}_{i,3}\in Q(\B G, \Aut)$ is defined using the same
quasi-morphism in $Q(\B B_3, \B A_3)$ as
$\overline{\Phi}_{i,3,r}\in Q(\B G_r, \Aut)$.
Recall that the support of each $g_j$ is contained inside the disc
$\D_r$. Since $r=\frac{1}{k}\thinspace$, there exists a family of diffeomorphisms
$\{h_j\}_{j=1}^k$ in $\B G$, such that $h_j\circ g_j\circ h_j^{-1}$ and
$h_i\circ g_i\circ h_i^{-1}$ have disjoint supports for all different $i$ and
$j$ between $1$ and $k$. It follows from the definition of Calabi homomorphism,
see Remark \ref{rem:Calabi}, that there exists a family $\{g'_i\}_{i=1}^k$ of
\emph{autonomous} diffeomorphisms in $\B G$ such that

\begin{itemize}
\item
The diffeomorphisms $g'_i$ and $g'_j$ have disjoint supports for
$i\neq~j$, and the diffeomorphisms $g'_i$ and $h_j\circ g_j\circ h_j^{-1}$ have
disjoint supports for all $1\leq i,j\leq k$.
\item
For each $1\leq i\leq k$ we have $\C(g'_i)=\C(h_i\circ g_i\circ h_i^{-1})$.
\end{itemize}

Denote by $f_i:=h_i\circ g_i\circ h_i^{-1}\circ (g'_i)^{-1}$, and let $\B K:=\ker\C$.
Note that all of $f_i$ have disjoint supports, $\C(f_i)=0$, each $f_i$ lies in
$\B K$ and they generate a free abelian group of rank $k$.
Let
$$
\Psi\colon\B Z^k\to\B G
$$
where $\Psi(d_1,\ldots,d_k)=f_1^{d_1}\circ\ldots\circ f_k^{d_k}$.
It is obvious that $\Psi$ is a monomorphism whose image lies in $\B K$. In
order to complete the proof of the theorem it is left to show that $\Psi$ is a
bi-Lipschitz map, i.e. we are going to show that there exists a constant
$A\geq 1$ such that
$$
A^{-1}\sum_{i=1}^k |d_i|\leq\|f_1^{d_1}\circ\ldots\circ f_k^{d_k}\|_{\OP{Aut}}
\leq A\sum_{i=1}^k |d_i|\thinspace.
$$
We have the following equalities
$$
\overline{\Phi}_{i,3}(f_j)
=\overline{\Phi}_{i,3}(h_j\circ g_j\circ h_j^{-1}\circ (g'_j)^{-1})
=\overline{\Phi}_{i,3}(g_j)+\overline{\Phi}_{i,3}((g'_j)^{-1})
=\overline{\Phi}_{i,3}(g_j)
=\delta_{ij}.
$$
The second equality follows from the fact that every homogeneous quasi-morphism
is invariant under conjugation and it behaves as a homomorphism on every pair
of commuting elements. The third equality follows from the fact that
$(g'_j)^{-1}$ is an autonomous diffeomorphism and
$\overline{\Phi}_{i,3}\in Q(\B G,\Aut)$, and the forth equality is \eqref{eq:delta-ij}.
Since all $f_i$
commute with each other and $\overline{\Phi}_{i,3}(f_j)=\delta_{ij}$, we obtain
$$
\|f_1^{d_1}\circ\ldots\circ f_k^{d_k}\|_{\OP{Aut}}
\geq \frac{|\overline{\Phi}_{i,3}(f_1^{d_1}\circ\ldots\circ
f_k^{d_k})|}{D_{\overline{\Phi}_{i,3}}}
=\frac{|d_i|}{D_{\overline{\Phi}_{i,3}}}\thinspace,
$$
where
$D_{\overline{\Phi}_{i,3}}$ is the defect of the quasi-morphism
$\overline{\Phi}_{i,3}$. The defect $D_{\overline{\Phi}_{i,3}}\neq~0$ because
each $\overline{\Phi}_{i,3}\in Q(\B G,\Aut)$ and hence it is not a
homomorphism. We denote by $\mathfrak{D}_k:=\max\limits_i
D_{\overline{\Phi}_{i,3}}$ and obtain the following inequality
\begin{equation}\label{eq:L-norms}
\|f_1^{d_1}\circ\ldots\circ f_k^{d_k}\|_{\OP{Aut}}
\geq(k\cdot\mathfrak{D}_k)^{-1}\sum_{i=1}^k |d_i|\thinspace .
\end{equation}
Denote by $\mathfrak{M}_f:=\max\limits_i\|f_i\|_{\OP{Aut}}$. Now we
have the following inequality
\begin{equation}\label{eq:easy-ineq}
\|f_1^{d_1}\circ\ldots\circ f_k^{d_k}\|_{\OP{Aut}}
\leq \sum_{i=1}^k |d_i|\cdot\|f_i\|_{\OP{Aut}}
\leq \mathfrak{M}_f\cdot \sum_{i=1}^k |d_i|\thinspace .
\end{equation}
Inequalities \eqref{eq:L-norms} and \eqref{eq:easy-ineq} conclude the proof of the theorem.
\qed

\begin{rem}
In fact, the proof of Theorem \ref{T:main} shows that for each
$k\in~\B N$ there exists a bi-Lipschitz embedding of $\B Z^k$ into $(\B G, \B d_{\Aut})$ with the image contained in a $C^0$-neighborhood of the identity diffeomorphism in $\B G$.
\end{rem}

\section{A relation between $Q(\B G)$ and $Q(\B F_2)$}
\label{sec-relation}
Let $\B F_2$ denote the free group on two generators, and let $\s_1$ and $\s_2$
denote the Artin generators of $\B B_3$. The center of both $\B B_3$ and
$\B P_3$ is a cyclic group generated by an element
$\Delta=\eta_{2,3}\cdot\eta_{3,3}$. One can show
that $\B P_3$ is generated by $\s_1^2,\s_2^2, \Delta$ and
$\B P_3\cong \B F_2\times Z(\B P_3)$, where $\B F_2=\langle\s_1^2,\s_2^2\rangle$.
In what follows we describe a monomorphism from $\B F_2$ to $\B G$ and study
the induced map from $Q(\B G)$ to $Q(\B F_2)$, which is infinite dimensional
by the theorem of Brooks \cite{Br}.

Let $U_1,U_2,U_3\subset \B D^2$ be open subsets
each diffeomorphic to a disc, such that $\area(U_i)\geq \frac{\pi}{4}$. We also
require that $z_i\in U_i$, where $z=(z_1,z_2,z_3)$ is a basepoint for
$\pi_1(\X_3(\D^2))\cong\B P_3$. For pairs $(U_1,U_2)$ and $(U_2,U_3)$ let
$W_{12}\subset V_{12}$ and $W_{23}\subset V_{23}$ be pairs of two open subsets
of $\D^2$, each diffeomorphic to a disc, such that $U_1\cup U_2\subset W_{12}$,
$U_2\cup U_3\subset W_{23}$, $V_{12}\cap U_3=\emptyset$ and $V_{23}\cap
U_1=\emptyset$. Let $\{h_t\}$ be a path in $\B G$ which rotates $W_{12}$ once,
and is identity on the outside of $V_{12}$ and on a small neighborhood of
$\partial V_{12}$. Similarly, let $\{h'_t\}$ be a path in $\B G$ which rotates
$W_{23}$ once, and is identity on the outside of $V_{23}$ and on a small
neighborhood of $\partial V_{23}$.

Let $U:=U_1\cup U_2\cup U_3$ and let
$\B G_U$ be the subgroup of $\B G$ which consists of diffeomorphisms
which \emph{preserve pointwise} the set  $U$. Let
$$Tr\colon\B G_U\to \B P_3\thinspace,$$
where $Tr(g)$ is the homotopy class of the loop $(g_t(z_1),g_t(z_2),g_t(z_3))$
in $\X_3(\D^2)$. Here $\{g_t\}_{t=0}^1$ is any isotopy from the identity map to
$g$. Since the map $Tr$ is a homomorphism, which sends $h_1$ to $\s_1^2$ and
$h'_1$ to $\s_2^2$, the diffeomorphisms $h_1$ and $h'_1$ generate a free group
in $\B G$. Let
$$s_U\colon\B F_2\to\B G$$
be a monomorphism, where $s_U(\s_1^2)=h_1$ and $s_U(\s_2^2)=h'_1$. Denote by $a_i=\area(U_i)$ and
$a=\area(U)$.

\begin{thm}\label{T:inf-dim-limit}
Let $Q(s_U)\colon Q(\B G,\Aut)\to Q(\B F_2)$ be the map induced
by the homomorphism $s_U$. Then
$$\lim_{a\to\pi}\dim(\Im(Q(s_U)))=\infty\thinspace.$$
\end{thm}

\begin{proof}
Let $N\in \B N$. We are going to show that there exists $\varepsilon>0$ such
that whenever $|a-\pi|<\varepsilon$ we have $\dim(\Im(Q(s_U)))\geq N$. Notice
that every $\varphi\in Q(\B B_3, \B A_3)$ vanishes on $\Delta$. Since
$\dim(Q(\B B_3, \B A_3))=\infty$ and $\B P_3$ is a subgroup of finite index in
$\B B_3$, it follows from Lemma \ref{L:abstract} that there exists a family of quasi-morphisms
$\{\varphi_i\}_{i=1}^N$ in $Q(\B B_3, \B A_3)$ and a family of braids
$\{\b_i\}_{i=1}^N$ which are words in $\s_1^2,\s_2^2$, such that
$\varphi_i(\b_j)=\delta_{ij}$. Denote by $g_{U,i}:=s_U(\b_i)$, i.e. each
$g_{U,i}$ is a time-one map of an isotopy $g_{t,i}$ which is a composition of a
number of isotopies $h_t$ and $h'_t$ that twist $U_j$'s in the form of the
braid $\b_i$. We are going to show that there exists $\varepsilon>0$, such that
if $|a-\pi|<\varepsilon$ then the matrix
$$
M_{N\times N}:=\left(\G_3(\varphi_i)(g_{U,j})\right)_{1\leq i\leq j\leq N}
$$
is non-singular, where $\G_3\colon Q(\B B_3, \B A_3)\to Q(\B G,\Aut)$. This will
imply that $\{Q(s_U)(\G_3(\varphi_i))\}_{i=1}^N$ are linearly independent in $Q(\B F_2)$.

It is easy to show that there exists $\varepsilon'>0$, such that each $N\times N$
matrix with entries $m_{ij}$ is non-singular provided that $1<m_{ii}<~12$
and $|m_{ij}|<\varepsilon'$ for all $i\neq j$. Denote by
$\X_3(U):=\X_3(\bigcup\limits_{i=1}^3 U_i)$. Since each
$\varphi_i\in Q(\B B_3, \B A_3)$ is invariant under conjugation in $\B B_3$
and vanishes on the braid $\eta_{3,3}$ we have
$$
\int\limits_{\X_3(U)}\lim_{p\to +\infty}
\frac{\varphi_i(\g(g_{U,j}^p;x))}{p}\thinspace dx=
\begin{cases}
0 &\text{ if } i\neq j\\
6a_1\cdot a_2\cdot a_3 &\text{ if } i=j\thinspace.
\end{cases}
$$
It follows that
\begin{equation}\label{eq:cases-inf-dim-thm}
\G_3(\varphi_i)(g_{U,j})=
\begin{cases}
\int\limits_{X_3(\B D^2)\setminus X_3(U)}\lim\limits_{p\to +\infty}
\frac{\varphi_i(\g(g_{U,j}^p;x))}{p}\thinspace dx
&\text{ if } i\neq j\\
6a_1\cdot a_2\cdot a_3 +\int\limits_{X_3(\B D^2)\setminus X_3(U)}
\lim\limits_{p\to +\infty}\frac{\varphi_i(\g(g_{U,i}^p;x))}{p}\thinspace dx
&\text{ otherwise }.
\end{cases}
\end{equation}
For $x\in \X_3(\D^2)$ denote by $cr(g_{U,i}^p;x)$ the length of the word in generators $\s_1, \s_2$, which represents the braid $\g(g_{U,i}^p;x)$ and is given by $p$ concatenations of flows $g_{t,i}$. Let
$$cr(\b_i):=cr(g_{U,i};z)\quad\text{ and }\quad \mathfrak{M}_{cr}:=\max_{1\leq i\leq N}{cr(\b_i)}\thinspace,$$
where $z=(z_1,z_2,z_3)$. It follows from the construction of diffeomorphisms $g_{U,i}$, that for each $x\in \X_3(\D^2)$ and $1\leq i\leq N$ we have
$$\lim_{p\to +\infty}\frac{cr(\g(g_{U,i}^p;x))}{p}\leq \mathfrak{M}_{cr}\thinspace .$$
For each $\g\in\B B_3$ denote by $l(\g)$ the word length of $\g$ with respect to the generating set $\s_1,\s_2$. Since each $\varphi_i$ is a homogenous quasi-morphism that vanishes on $\s_1,\s_2$, we obtain
$$|\varphi_i(\g)|\leq D_{\varphi_i}\cdot l(\g)\thinspace .$$
Denote by
$$\mathfrak{M}_D:=\max_{1\leq i\leq N}{D_{\varphi_i}}\thinspace.$$
It follows that for each $x\in \X_3(\D^2)$ and $1\leq i,j\leq N$ we have
\begin{eqnarray*}
&\lim\limits_{p\to +\infty}\frac{|\varphi_i(\g(g_{U,j}^p;x))|}{p}\leq \mathfrak{M}_D\lim\limits_{p\to +\infty}\frac{|l(\g(g_{U,j}^p;x))|}{p}\\
&\leq \mathfrak{M}_D\lim\limits_{p\to +\infty}\frac{cr(\g(g_{U,j}^p;x))}{p}\leq \mathfrak{M}_D\cdot \mathfrak{M}_{cr}\thinspace .
\end{eqnarray*}
Take $\varepsilon>0$, such that
$$\mathfrak{M}_D\cdot \mathfrak{M}_{cr}\cdot\area\left(\X_3(\D^2)\setminus\X_3(U)\right)<\min\left\{\frac{1}{10},\varepsilon'\right\}\thinspace.$$
Equality \eqref{eq:cases-inf-dim-thm} yields
\begin{equation*}
|\G_3(\varphi_i)(g_{U,j})|\leq \varepsilon' \quad\text{ if }\quad i\neq j \quad\text{ and }\quad 1\leq\G_3(\varphi_i)(g_{U,i})\leq 12,
\end{equation*}
hence the matrix $M_{N\times N}$ is non-singular and the proof follows.
\end{proof}

\section{Comparison of bi-invariant metrics on $\B G$ and other comments}
\label{sec-comparision}

\subsection{The Hofer metric}\label{SS:hofer}
The most famous metric on the group of
Hamiltonian diffeomorphisms of a symplectic manifold $\Mo$
is the Hofer metric, see \cite{Ho,LM}. The associated norm is defined by
$$
\|f\|_{\OP{Hofer}}:=\inf_{F_t}\int_0^1\OP{osc}(F_t)dt,
$$
where $F_t$ is a compactly supported Hamiltonian function generating
the Hamiltonian flow $f_t$ from the identity to $f=f_1$.
The oscillation norm is defined by
$\OP{osc}(F)=\max_M F - \min_M F$.

\begin{ex}
Let $f\in \B G$ be a diffeomorphism
generated by a time independent and non-negative Hamiltonian
function $F$. It implies that all powers of $f$ are
also autonomous and hence $\|f^n\|_{\OP{Aut}} =~1$ for all $n\in \B Z$.
On the other hand, the Calabi homomorphism is positive
on $f$ and hence $\|f^n\|_{\OP{Hofer}}\geq \OP{const}|n| \C(f)$,
for some positive constant.

Also, $\|f^{1/n}\|_{\OP{Aut}}=1$ but
$\lim\limits_{n\to \infty} \|f^{1/n}\|_{\OP{Hofer}}=0$.
Here $f^{1/n}$ is the unique diffeomorphism in
the flow generated by $F$ such that its $n$-th power
is equal to $f$.
This shows that the identity homomorphism between
the autonomous metric and the Hofer metric is
not Lipschitz in neither direction.
\end{ex}

\subsection{The restricted autonomous metric}
\label{SS:restricted}

Let $S_r\subset \Ham\Mo$ be the set of autonomous diffeomorphisms generated by Hamiltonian functions
with the $L^{\infty}$-norm bounded by $r>0$. This set is invariant under conjugations and hence
the corresponding word metric is bi-invariant. We call it the restricted autonomous metric and
denote the corresponding norm by $\|f\|_r$. For all $r$ these metrics are Lipschitz equivalent.
Indeed, it is easy to check that if $r\leq R$ then
$$
\|f\|_R\leq \|f\|_r \leq \lceil R/r\rceil \|f\|_R
$$
for all $f\in \Ham\Mo$. Moreover, we have
that $\|f\|_{\OP{Aut}}\leq \|f\|_r$ for every~$r$.
This trivially implies that the main results
of the paper hold for the restricted autonomous
metric.

Let $f\in \B G$ be such that
$f=h_1\circ \dots \circ h_k$ with each $h_i$ is autonomous generated by a Hamiltonian
$H_i$ of the oscillation norm smaller than $r$.
\begin{eqnarray*}
\C(f)&=& \int_0^1dt \int_{\B D^2} F_t\omega\\
&=& \int_0^1 dt \int_{\B D^2} \sum_{i=1}^k
H_i ((h_{1,t}\circ \dots \circ h_{i-1,t})^{-1})\omega\\
&\leq&  \sum_{i=1}^k \OP{osc}H_i \leq k r
\end{eqnarray*}
We thus obtain the following estimate
$$
\C(f)/r \leq \|f\|_r
$$
which proves that the restricted autonomous norm is not equivalent to the autonomous norm for there
are autonomous diffeomorphisms with arbitrarily big Calabi invariant.

\subsection{Fragmentation metrics}
\label{SS:fragment}
Let $U\subset M$ be a set with nonempty interior. The fragmentation metric
${\bf d}_U$ is a word metric defined with respect to the generating set
consisting of diffeomorphisms conjugated to ones supported in $U$. Such a set
is invariant under conjugations by construction and hence the fragmentation
metric is bi-invariant.

It follows from the proof of Theorem \ref{T:main} that there
is a diffeomorphism $f$ supported in the set $U$ such that
$f$ has arbitrarily big autonomous norm. Clearly, the
fragmentation norm of $f$ is equal to one.

\begin{ex}\label{E:frag}
Suppose that $U\subset \B D^2$ is a disc of radius $1/2$.
According to Biran, Entov and Polterovich \cite{BEP}, the
space of homogeneous Calabi quasi-morphisms on
$\Ham(\B D^2)$ is infinite dimensional. The Calabi property
means that the restriction of a quasi-morphism
to the subgroup of diffeomorphisms supported on a
displaceable subset is equal to the Calabi homomorphism.

Consider the subgroup $\B K = \ker \C \subset \B G$.
It is generated
up to conjugation by diffeomorphisms supported in $U$
and hence the fragmentation metric is defined on $\B K$.
(In the next section we explain that this metric is
equal to the autonomous metric induced from $\B G$).
Let $q\colon \B K \to \B R$ be a Calabi quasi-morphism.
Since it is trivial on the generators it is Lipschitz
with respect to the fragmentation norm.

It follows from the proof of Theorem 2.3 in \cite{BEP}
that there is an autonomous diffeomorphism $f\in \B K$
and a homogeneous Calabi
quasi-morphism $q\colon \B K\to \B R$ such that
$q(f)>0$. This implies that $f^n$ can have arbitrarily
big fragmentation norm. Its autonomous norm is
equal to one.
\end{ex}

\subsection{The kernel of the Calabi homomorphism}
\label{SS:ker_calabi}
This is a  remark on the geometry of the inclusion
$\ker \C = \B K \to \B G$ with respect to the autonomous
metric. Observe that the kernel of the Calabi homomorphism
is generated by autonomous diffeomorphisms and
let $\|g\|_{\OP{Aut^{\prime}}}$ denotes the corresponding
autonomous norm of $g\in \B K$.

\begin{lem}\label{L:ker_calabi}
Let ${\bf i}\colon \B K\to \B G$ be the inclusion.
Then
$$
\|g\|_{\OP{Aut^{\prime}}} = \|{\bf i}(g)\|_{\OP{Aut}}
$$
for every $g\in \B K$.
\end{lem}

\begin{proof}
By definition for each $g\in\B K$ we have
$\|\B i(g)\|_{\OP{Aut}}\leq\|g\|_{\OP{Aut}^{\prime}}$.
Let $g\neq\Id$ and suppose that $\|\B i(g)\|_{\OP{Aut}}=m$. It means that the diffeomorphism
$g=h_1\circ\ldots\circ h_m$ for some autonomous diffeomorphisms $h_i$.
It is straightforward to construct autonomous
diffeomorphisms $f_1,\ldots,f_{m-1}$ such that
\begin{itemize}
\item
the diffeomorphisms $f_i$ and $f_j$ have disjoint supports for $i\neq~j$,
and the diffeomorphisms $f_i$ and $h_j$ have disjoint supports for all
$1\leq i\leq m-1$ and $1\leq j\leq m$,
\item
$\C(f_1)=\C(h_1)$ and
$\C(f_i)=\C(f_{i-1}\circ h_i)$
for $2\leq i\leq m-1$.
\end{itemize}
For example, we can take autonomous diffeomorphisms $f_i$ disjointly supported away from the union of the supports of $h_j$'s and with
appropriate values of the Calabi homomorphism.  We can write $g$ as follows
$$
g=(h_1\circ f_1^{-1})\circ(f_1\circ h_2\circ f_2^{-1}) \circ\ldots\circ
(f_{m-2}\circ h_{m-2}\circ f_{m-1}^{-1})\circ(f_{m-1}\circ h_m).
$$
Note that
$$\C(h_1\circ f_1^{-1})= \C(f_{i-1}\circ h_i\circ f_i^{-1})=0$$
for $2\leq i\leq m-1$, and $\C(f_{m-1}\circ h_m)=0$ because $\C(g)=0$.
Since each $h_i$ commutes with each $f_j$ and each $f_i$ commutes with each
$f_j$, the diffeomorphisms $h_1\circ f_1^{-1}$,
$f_{i-1}\circ h_i\circ f_i^{-1}$ for $2\leq i\leq m-1$, and
$f_{m-1}\circ h_m$ are autonomous diffeomorphisms
which finishes the proof.
\end{proof}

\subsection*{Acknowledgments}
Both authors would like to thank the anonymous referee for careful reading of our paper and for his/her helpful comments and remarks. This work has been done during the first author stay in Aberdeen and in Mathematisches Forschungsinstitut Oberwolfach.  The first author wishes to
express his gratitude to both institutes. He was supported by ESF grant number $4824$ and by the
Oberwolfach Leibniz fellowship. The visit in Aberdeen was supported by the CAST network.

\vspace{3mm}
Department of Mathematics, Vanderbilt University, Nashville, TN\\
\emph{E-mail address:} \verb"michael.brandenbursky@vanderbilt.edu"
\vspace{3mm}

University of Aberdeen and University of Szczecin\\
\emph{E-mail address:} \verb"kedra@abdn.ac.uk"

\end{document}